\documentclass[12pt]{amsart}
\usepackage{latexsym, url}
\usepackage{amsthm}
\usepackage{amsmath}
\usepackage{amsfonts}
\usepackage{amssymb}
\usepackage[dvips]{graphicx}
\usepackage[all]{xypic}
\input xy
\xyoption{all}

\newtheorem{theo}{Theorem}[section]
\newtheorem{lemma}[theo]{Lemma}
\newtheorem{assume}[theo]{Assumption}
\newtheorem{propo}[theo]{Proposition}
\newtheorem{defi}[theo]{Definition}
\newtheorem{coro}[theo]{Corollary}
\newtheorem{rem}[theo]{Remark}
\newtheorem{exam}[theo]{Example}
\newtheorem{exams}[theo]{Examples}
\newcommand\ksst{\preceq_\ck}
\newcommand\lsst{\subseteq_L}
\newcommand\lsk{LS(\ck)}
\newcommand\lsdk{LS^d(\ck)}
\newcommand\Ind{\operatorname{Ind}}

\newcommand\Mod{\operatorname{Mod}}

\newcommand\op{\operatorname{op}}

\newcommand\metr{\operatorname{\bf d}}
\newcommand\gat{\operatorname{ga-S}}

\newcommand\Set{\operatorname{\bf Set}}
\newcommand\Met{\operatorname{\bf Met}}

\newcommand\Str{\operatorname{\bf Str}}
\newcommand\mStr{\operatorname{\bf mStr}}
\newcommand\Ab{\operatorname{\bf Ab}}

\newcommand\Lin{\operatorname{\bf Lin}}

\newcommand\dc{\operatorname{dc}}

\newcommand\gatp{\operatorname{ga-tp}}

\newcommand\colim{\operatorname{colim}}

\newcommand\monst{\mathfrak {C}}
\newcommand\reals{\mathbb {R}}
\newcommand\ca{\mathcal {A}}

\newcommand\cc{\mathcal {C}}

\newcommand\cf{\mathcal {F}}

\newcommand\ck{\mathcal {K}}
\newcommand\cl{\mathcal {L}}

\newcommand\crr{\mathcal {R}}
\newcommand\cs{\mathcal {S}}

\date{April 7, 2016}
 
\begin{document}
\title[Metric abstract elementary classes as accessible categories]
{Metric abstract elementary classes as accessible categories}
\author[M. Lieberman and J. Rosick\'{y}]
{M. Lieberman and J. Rosick\'{y}}
\thanks{Supported by the Grant Agency of the Czech Republic under the grant 
               P201/12/G028.} 
\address{
\newline M. Lieberman\newline
Department of Mathematics and Statistics\newline
Masaryk University, Faculty of Sciences\newline
Kotl\'{a}\v{r}sk\'{a} 2, 611 37 Brno, Czech Republic\newline
lieberman@math.muni.cz\newline
\newline J. Rosick\'{y}\newline
Department of Mathematics and Statistics\newline
Masaryk University, Faculty of Sciences\newline
Kotl\'{a}\v{r}sk\'{a} 2, 611 37 Brno, Czech Republic\newline
rosicky@math.muni.cz
}
 
\begin{abstract}
We show that metric abstract elementary classes (mAECs) are, in the sense of \cite{LR}, coherent accessible categories with directed colimits, with concrete $\aleph_1$-directed colimits and concrete monomorphisms.  More broadly, we define a notion of {\it $\kappa$-concrete AEC}---an AEC-like category in which only the $\kappa$-directed colimits need be concrete---and develop the theory of such categories, beginning with a category-theoretic analogue of Shelah's Presentation Theorem and a proof of the existence of an Ehrenfeucht-Mostowski functor in case the category is large.  For mAECs in particular, arguments refining those in \cite{LR} yield a proof that any categorical mAEC is $\mu$-$\metr$-stable in many cardinals below the categoricity cardinal.
\end{abstract} 
\keywords{}
\subjclass{}

\maketitle

\section{Introduction}

This paper may be regarded as an addition to the expanding literature on the interactions between category theory and abstract model theory and, in particular, as an extension of the results of \cite{LR} from abstract elementary classes (AECs) to metric abstract elementary classes (mAECs).  The latter may be thought of as a kind of amalgam of AECs with the program of continuous logic.  Continuous logic has its origins in the work of Chang and Keissler, and has subsequently been developed by Henson, Iovino, Usvyatsov and Ben-Yaacov, among others, always with an eye toward applications of model theory to structures arising in analysis.  Thus in an mAEC, as opposed to an AEC, the structures under consideration typically have as their underlying universe of discourse not a discrete set but a complete metric space.

In \cite{LR}, the authors develop a hierarchy of accessible categories with additional structure, resulting, ultimately, in a precise characterization of AECs as concrete categories.  Roughly speaking, the hierarchy is as follows, assuming throughout that all morphisms are monomorphisms:
\begin{enumerate}\item $\ck$ is an accessible category (see \cite{AR}, \cite{MP}).
\item $\ck$ is an accessible category with directed colimits (see \cite{BR}, \cite{R}).
\item $(\ck, U)$ is an accessible category with concrete directed colimits and concrete monomorphisms, i.e. $\ck$ is equipped with a faithful functor $U:\ck\to\Set$ that preserves directed colimits and monomorphisms.
\item $(\ck, U)$ is a coherent accessible category with concrete directed colimits and concrete monomorphisms, where ``coherence'' is a property of $U$ corresponding to the coherence axiom for AECs.
\item $(\ck, U)$ is a coherent accessible category with concrete directed colimits and concrete monomorphisms, and satisfies the iso-fullness condition described in Remark 3.5 in \cite{LR}---such a category is equivalent to an AEC.\end{enumerate}
Certain essential results from the theory of AECs are shown to hold at greater levels of generality: categories of the form (2) satisfy a presentation theorem generalizing that of Shelah and, if large, admit a robust EM-functor.  Categories of the form (3) allow the development of Galois types and satisfy a generalization of Boney's theorem on tameness under the assumption of a proper class of strongly compact cardinals (see \cite{B}).  Categories of the form (4) satisfy the essential technical condition that Galois saturation corresponds to, in AEC terms, model-homogeneity, and support the development of a fragment of classification theory.  

We show in Section~\ref{maecsareacc} that any mAEC $\ck$ is an accessible category with directed colimits, and note that, if we take $U:\ck\to\Set$ to be the usual underlying set functor, $\ck$ is coherent with concrete monomorphisms.  As is well known, though, directed colimits in $\ck$ need not be concrete: when taking the colimit of a chain of structures in $\ck$, we must, in general, take the completion of the union of the underlying sets.  This would seem to place us, at best, in type (2) above, which is already sufficient to give a presentation theorem and guarantee the existence of an EM-functor for a general mAEC---this is in itself a generalization of \cite{HH}, the results of which hold only in the homogenous case.  As we note in Remark~\ref{ldircolimsconcr}, however, mAECs do have concrete $\aleph_1$-directed colimits, which suggests that we may benefit from a generalization of the hierarchy of \cite{LR}, considering categories with concrete $\kappa$-directed colimits for some $\kappa$.  In case a category of this form has (not necessarily concrete) directed colimits, is coherent, has concrete monomorphisms, and is suitably replete and iso-full---the conditions of (4) above---we call it a {\it $\kappa$-concrete AEC}, or {\it $\kappa$-CAEC} for short.  Incidentally, there is an alternative option already being pursued in, e.g., \cite{muaecs}, namely to consider classes of structures which are indistinguishable from AECs, except insofar as they are only required to have $\kappa$-directed colimits and satisfy a subtle weakening of the usual L\"owenheim-Skolem axiom: this notion, \textit{$\kappa$-AEC}, is more general than the one we investigate here.

We introduce the definition of $\kappa$-CAEC in Section~\ref{kaec}, and develop a few basic results for such categories.  Most importantly, we show that the results of \cite{LR} can be generalized to this context, with only minor modifications.  In fact, many arguments go through without change: an analogue of Shelah's Presentation Theorem and the existence of EM-functors for large $\kappa$-CAECs follow immediately.  As we will see, though, by contrast to case (4) above, if $\ck$ is a $\kappa$-CAEC, the functor $U:\ck\to\Set$ need not preserve all sizes $\lambda>\kappa$, but rather preserves $\lambda$-presentable objects for $\lambda\triangleright\kappa$, where $\triangleright$ is the relation described in \cite{MP}.  The end result is a slight weakening of the results of Sections 6 and 7 of the earlier paper (concerning, respectively, the equivalence of category-theoretic saturation and Galois-saturation, and stability and the existence of saturated models in categorical AECs) which nonetheless hold for general $\kappa$-CAECs.  In the case of mAECs, in particular, they are stronger than existing results along these lines.

\section{Metric AECs}\label{maecs}

As mentioned above, we work in the context of metric AECs (mAECs), as considered in \cite{HH}, \cite{VZ}, and \cite{Z}: classes in which the structures have complete metric spaces rather than sets as their sorts, and where the interpretations of the function and relation symbols are required to behave well with respect to the appropriate metrics.  To be precise, let $L$ be a language with sorts $\cs\cup\{\reals\}$, function symbols $\cf\cup\{d_\sigma\}_{\sigma\in \cs}$, relation symbols $\crr$, and constant symbols $\cc$.

The $d_\sigma$ are to be interpreted as $\reals$-valued metrics on the sorts $\sigma$.  The other symbols have prescribed arities as well: each $c\in\cc$ is of sort $\nu(c)$, each $R\in\crr$ is a predicate on a product sort $\nu(R)_1\times {\nu(R)_2}\times\dots\times {\nu(R)_n}$, and so on.

\begin{defi}\label{mlstruct} {\em A many-sorted metric $L$-structure is given by a tuple of interpretations of the symbols in $L$,
$$( \{ (\sigma^M,d_\sigma^M) \}_{\sigma\in\cs}, \reals^M, \{ R^M \}_{R\in\crr}, \{ F^M \}_{F\in\cf}, \{ c^M\}_{c\in\cc} )$$
Here each $(\sigma^M,d_\sigma^M)$ is a complete metric space, $\reals^M$ is a copy of the real numbers, and each $c^M$ is an element of ${\nu(c)}^M$.  The interest lies in the functions and relations:
\begin{enumerate}\item Each $F\in\cf$ is interpreted as a function 
$$F^M:\nu(F)_1^M\times\dots\times \nu(F)_m^M\to \nu(F)^M$$
that is continuous (with respect to the product metric on the left side). 
\item Each $R\in\crr$ is interpreted as a continuous predicate, 
$$R^M:\nu(R)_1^M\times\dots\times\nu(R)_n^M\to [0,1]$$  
(again, with respect to the product metric on the left side).
\end{enumerate}}\end{defi}

Note that the product metric is the maximum metric and continuity then means that the convergence of sequences is preserved.
This follows \cite{VZ} and \cite{Z} while \cite{HI} uses uniform continuity. In fact, the fine-grained distinction between the two versions does not register in the category-theoretic characterization of mAECs that we develop in Sections 3 and 4---the results of the ensuing sections apply equally in both cases.

We form a category of metric $L$-structures, $\mStr(L)$, by taking the morphisms to be the metric $L$-structure embeddings, i.e. maps $f:M\to N$ satisfying
\begin{enumerate}\item For all $c$, $f(c^M)=c^N$.
\item For all $F$ and $\bar{a}$ in $M$ of appropriate arity, $f(F^M(\bar{a}))=F^N(f(\bar{a}))$.
\item For all $R$, $\bar{a}$ in $M$ of appropriate arity and $k\in [0,1]$, $d(\bar{a},[R^M]^{-1}(k))=d(f(\bar{a}),[R^N]^{-1}(k))$
(with the product metric).
\end{enumerate}

Notice that, because we have included the metrics in the language, (2) guarantees that any $\mStr(L)$-map is a sortwise isometry.

Just as AECs are traditionally axiomatized as a subclasses---better, subcategories---of an ambient category $\Str(L)$ of discrete structures, mAECs are axiomatized as subcategories of $\mStr(L)$.  As with AECs, the concern is to refine the notion of substructure/embedding, and the axioms by which this is achieved are almost identical.  The only essential changes are that we must, in general, take the completions of unions of chains, and that density character takes the place of size in the metric context.  Recall:

\begin{defi} {\em The density character of a complete metric space $X$, denoted $\dc(X)$, is the cardinality of the smallest dense subset of $X$.  We define the density character of a subset $A\subseteq X$ to be the density character of its completion, i.e. $\dc(A)=\dc(\overline{A})$.}\end{defi}

Noting that each metric $L$-structure $M$ is not a metric space but rather a collection of metric spaces, one for each sort, we define $\dc(M)$ to be the sum of the density characters of its sorts.

\begin{defi} {\em Let $\ck$ be a class of metric $L$-structures in the sense of Definition~\ref{mlstruct}, and let $\ksst$ be a partial order on $\ck$.  We say that $(\ck,\ksst)$ is a metric AEC is $\ksst$ refines the usual metric substructure relation and the following additional conditions hold:
\begin{enumerate}\item $\ck$ and $\ksst$ are closed under isomorphism.
\item (Colimits of chains) If $\langle M_i\,|\,i<\lambda\rangle$ is a $\ksst$-increasing chain, then
\begin{enumerate}\item the function and predicate symbols in $L$ can be extended uniquely from $\bigcup_{i<\lambda} M_i$ to its completion in such a way that $\overline{\bigcup_{i<\lambda} M_i}\in\ck$,
\item for all $i<\lambda$, $M_i\ksst\overline{\bigcup_{i<\lambda} M_i}\in\ck$, and
\item if $M_i\ksst N$ for all $i<\lambda$, then $\overline{\bigcup_{i<\lambda} M_i}\in\ck\ksst N$.\end{enumerate}
\item (Coherence) If $M_1\lsst M_2\ksst M_3$ and $M_1\ksst M_3$, then $M_1\ksst M_2$.
\item (L\"owenheim-Skolem) There exists an infinite cardinal $\lsdk$ such that for any $M\in\ck$ and subset $A\subseteq M$, there is an $N$ in $\ck$ with $\dc(N)\leq \dc(A)+\lsdk$ such that $A\subseteq N\ksst M$.
\end{enumerate}}
\end{defi}

\begin{rem} {\em AECs correspond to the special case in which all of the metrics are discrete.}\end{rem}

\begin{defi} {\em Given $f:M\to N$, $M$ and $N$ in $\ck$, we say that $f$ is a \textit{$\ck$-embedding} if $f[M]\ksst N$.}\end{defi}

The axiom concerning completions of unions of chains guarantees that $\ck$ has colimits of $\ksst$-chains which, by Corollary 1.7 in \cite{AR}, is equivalent to having directed colimits of $\ck$-embeddings.  Note that closure of mAECs under directed colimits is also proved directly as Corollary 1.2.6 in \cite{Z}.  For emphasis:

\begin{rem} {\em Any mAEC $\ck$ has arbitrary directed colimits.}\end{rem}

\begin{rem}\label{badstructs} {\em Closure under directed colimits is a remarkably strong assumption on an mAEC $\ck$ and its embeddings.  In particular, we cannot hope that a general category of metric structures $\mStr(L)$ will have directed colimits, as we see in the following example of \cite{kirby}. 

Consider the one sorted language with a single unary function symbol $f$.  Consider the chain of $\mStr(L)$-stuctures $\langle M_n\rangle_{n<\omega}$, where each $M_n$ has underlying set $\{0\}\cup\{ 1/k \,|\, 0<k\leq n+1 \}$ equipped with the metric inherited from $\reals$, and $f$ is interpreted as the characteristic function of the subset of nonzero elements: $f^{M_i}(0)=0$, but has value $1$ otherwise.  The union of the underlying sets $\{0\}\cup\{1/k\,|\,1<k<\omega\}$ is itself complete, and we are forced to take $f^{M}=\bigcup_{n<\omega} f^{M_n}$, which is again the characteristic function of the subset of nonzero elements.  This function is not continuous, meaning that $M$ cannot belong to $\mStr(L)$.

In fact, as noted in \cite{kirby}, this example also shows that categories of metric structures need not have all directed colimits even under the stronger assumption that the interpretations of function and predicate symbols are not merely continuous, but {\em uniformly continuous}.  The still more restrictive case of {\em contractions} is considered in Example~\ref{kaecexams}.
}\end{rem}

It is important to note that these directed colimits are, in general, not concrete: the underlying set of the colimit of a $\ksst$-increasing chain will be the completion of the chain's union, which need not correspond to the union itself.  That is, if $U:\ck\to\Set$ is the usual underlying set functor, 

\begin{rem}\label{dircolimsnonconcr} {\em In general, $(\ck,U)$ will not have concrete directed colimits.}\end{rem}

Given any uncountable regular cardinal $\lambda$, however, the colimit of any $\ksst$-increasing $\lambda$-chain (or, indeed, any $\lambda$-directed system of $\ksst$-substructures) should have precisely the union as its underlying set.  We prove the parenthetical, assuming, for simplicity, that our structures are one-sorted: given a $\lambda$-directed system of $\ksst$-substructures $\langle M_i\,|\,i\in I\rangle$, consider $x\in\overline{\bigcup_{i\in I}M_i}$.  Then $x$ is the limit of a sequence $\langle x_n\,|\,n\in\omega\rangle$ in $\bigcup_{i\in I}M_i$ and, by $\lambda$-directedness of the union, this sequence actually lies in some $M_j$.  As $M_j$ is complete, $x\in M_j\subseteq\bigcup_{i\in I}M_i$, meaning that $\overline{\bigcup_{i\in I}M_i}=\bigcup_{i\in I}M_i$.  Hence 

\begin{rem}\label{ldircolimsconcr}{\em For any mAEC $\ck$, $(\ck,U)$ has concrete $\lambda$-directed colimits, for all uncountable regular $\lambda$.}\end{rem}

\section{Metric AECs as Accessible Categories}\label{maecsareacc}

Accessible categories were introduced in \cite{MP} as categories closely connected with categories of models of $L_{\kappa,\lambda}$
theories---essential background can be found in \cite{MP} and \cite{AR}, while concrete accessible categories are treated in \cite{BR} and \cite{LR}. Roughly speaking, an accessible category is one that is closed under certain directed colimits, and whose objects can be built via certain directed colimits from a set of small objects.  To be precise, we say that a category $\ck$ is $\lambda$-\textit{accessible}, $\lambda$ a regular cardinal, if it closed under $\lambda$-directed colimits (i.e. colimits indexed by a $\lambda$-directed poset) and contains, up to isomorphism, a set $\ca$ 
of $\lambda$-presentable objects such that each object of $\ck$ is a $\lambda$-directed colimit of objects from $\ca$. 

Here $\lambda$-presentability functions as a notion of size that makes sense in a general, i.e. non-concrete, category: we say an object $M$ is $\lambda$-\textit{presentable} if its hom-functor $\ck(M,-):\ck\to\Set$ preserves $\lambda$-directed colimits. Put another way, an object $M$ is $\lambda$-presentable if for any morphism $f:M\to N$ with $N$ a $\lambda$-directed colimit $\langle \phi_\alpha:N_\alpha\to N\rangle$, $f$ factors essentially uniquely through one of the $N_\alpha$, i.e. $f=\phi_\alpha f_\alpha$ for some $f_\alpha:M\to N_\alpha$. 

Recall that the \textit{presentability rank} of an object $K$ in a category $\ck$ is the smallest regular cardinal $\kappa$ such that $K$ is $\kappa$-presentable.  Following Lemma 4.2 in \cite{BR}, if $\ck$ is a $\lambda$-accessible category with directed colimits and $K\in\ck$ is not $\lambda$-presentable then the presentability rank $\kappa$ of $K$ is a successor cardinal, i.e. $\kappa=|K|^+$ for some cardinal $|K|$. We think of $|K|$ as the \emph{internal size of $K$ in $\ck$}, or simply the \emph{size of $K$}.

\begin{theo}\label{maecsacc} Let $\ck$ be an mAEC.  Then $\ck$ is $\lambda$-accessible with directed colimits for all uncountable regular cardinals $\lambda>\lsdk$.  Moreover, for any uncountable cardinal $\lambda$, an object $M\in\ck$ is of presentability rank $\lambda^+$ if and only if $\dc(M)=\lambda$.\end{theo}

The moreover clause amounts to the assertion that an object in an mAEC is of size $\lambda$ if and only if $\dc(M)=\lambda$, i.e. the category-theoretic notion of size matches up perfectly with density character.  We proceed by a series of easy lemmas, paralleling the proof of the analogous result for AECs in $\S 4$ of \cite{L}.  We again work with one-sorted structures---the many-sorted case follows easily.

\begin{lemma}\label{ldircolim}Let $\ck$ be an mAEC.  For any regular $\lambda>\lsdk$, each object of $\ck$ is a $\lambda$-directed colimit of its $\ksst$-substructures of density character less than $\lambda$.\end{lemma}
\begin{proof} Let $M\in\ck$, and let $\langle M_i\,|\,i\in I\rangle$ be the system of $\ksst$-substructures of $M$ of density character less than $\lambda$.  We wish to show that this system is $\lambda$-directed.  To that end, let $\{M_{i_\alpha}\,|\,\alpha<\nu<\lambda\}$.  As each $M_{i_\alpha}$ is of density character less than $\lambda$, they each contain a dense subset $X_{i_\alpha}$ of cardinality $\lambda$.  Let $X=\bigcup_{\alpha<\nu} X_{i_\alpha}$.  Notice that 
$$\dc(X)\leq|X|=\sum_{\alpha<\nu}|X_{i_\alpha}|<\lambda$$  
By the Downward L\"owenheim Skolem axiom, there is a model $M'\ksst M$ containing $X$ with $\dc(M')\leq \dc(X)+\lsdk <\lambda$.  So $M'=M_j$ for some $j\in I$.  Moreover, coherence implies that $M_{i_\alpha}\ksst M'$ for all $\alpha<\nu$; that is, the diagram is $\lambda$-directed, as claimed.\end{proof}

\begin{lemma}\label{dctopres} Let $\ck$ be an mAEC and $\lambda$ be an uncountable regular cardinal.  If $M\in\ck$ has $\dc(M)<\lambda$, it is $\lambda$-presentable.\end{lemma}
\begin{proof}Suppose that $\dc(M)<\lambda$, and that $f:M\to N$ is a $\ck$-embedding, with $N$ the colimit of a $\lambda$-directed diagram, $\langle \phi_{ij}:N_i\to N_j\,|\,i\leq j\in I\rangle$, with cocone maps $\phi_i:N_i\to N$, i.e.
$$\xymatrix{ M \ar[rrrr]^f & & & & N & \\
& & & N_i \ar[rr]_{\phi_{ij}}\ar[ur]^{\phi_i} & & N_j\ar[ul]_{\phi_j}}$$
As this colimit is $\lambda$-directed, it is concrete: $|N|=\bigcup_{i\in I}|\phi_i[N_i]|$, where, for emphasis, the union is $\lambda$-directed.  As $f$ is a $\ck$-embedding, and therefore an isometry, $\dc(f[M])=\dc(M)<\lambda$.  That is, there is a dense subset $X\subseteq f[M]$ with $|X|<\lambda$.  By $\lambda$-directedness of the union, $X\subseteq \phi[N_i]$ for some $i\in I$.  As $\phi_i[N_i]$ is complete, $f[M]=\overline{X}\subseteq \phi[N_i]$ and, by coherence, $f[M]\ksst \phi_i[N_i]$.  Then $\phi_i^{-1}\circ f$ is the desired factorization of $f$ through $\phi_i$.  This factorization is unique, as well: given another $g:N_i\to N$ with $\phi_i\circ g=f=\phi_i\circ(\phi_i^{-1}\circ f)$, the fact that $\phi_i$ is a monomorphism guarantees that $g=\phi_i^{-1}\circ f$.
\end{proof}

Lemmas~\ref{ldircolim} and~\ref{dctopres} and the remarks preceding Theorem~\ref{maecsacc} imply the first part of the theorem: any mAEC $\ck$ is $\lambda$-accessible with directed colimits for any uncountable regular cardinal $\lambda>\lsdk$. To complete the proof of the moreover clause, we need:

\begin{lemma}\label{prestodc} Let $\ck$ be an mAEC.  If $M\in\ck$ is $\lambda$-presentable, $\dc(M)<\lambda$.\end{lemma}
\begin{proof} Let $M\in\ck$ be $\lambda$-presentable.  Consider the identity map on $M$.  By Lemma~\ref{ldircolim}, we can express $M$ as the $\lambda$-directed colimit of its system of $\ck$-substructures of density character less than $\lambda$, $\langle N_i\ksst N_j\,|\,i\leq j\in I\rangle$.  Since $M$ is $\lambda$-presentable, the identity map factors through some $N_i$, or rather through the $\ck$-inclusion of $N_i$ into $M$ itself.  But, given that all $\ck$-embeddings are concrete monomorphisms, $M=N_i$, and we are done.
\end{proof}

An immediate consequence is that, for an uncountable regular cardinal $\lambda$, an object $M$ in an mAEC $\ck$ is $\lambda$-presentable if and only if 
$\dc(M)<\lambda$. The moreover clause of Theorem~\ref{maecsacc} follows: if $M\in\ck$ has $\dc(M)=\lambda >\aleph_0$ then $M$ is 
$\lambda^+$-presentable and $M$ cannot be $\mu$-presentable for $\mu\leq\lambda$---if so, it would would need to satisfy $\dc(M)<\mu$. On the other hand, let $M\in\ck$
have presentability rank $\lambda^+$ with $\lambda$ uncountable. Then $\dc(M)\leq\lambda$.  If $\dc(M)<\lambda$ then $\dc(M)<\mu\leq\lambda$ for some uncountable regular $\mu$.  This means that $M$ is $\mu$-presentable, contradicting the assumption that $M$ has presentability rank $\lambda^+$. This completes the proof of Theorem~\ref{maecsacc}.  
 
\begin{exam}\label{metric}
{
\em
Let $\Met$ be the category of complete metric spaces and isometric embeddings. Then $\Met=\mStr(L)$ where $L$ has one sort $S$ and a single function symbol $d$ for the metric. Since $LS^d(\Met)=\aleph_0$, $\Met$ is $\lambda$-accessible for any uncountable regular cardinal $\lambda$.
Complete metric spaces of cardinality $<\aleph_0$ have presentability rank $\aleph_1$, thus size $\aleph_0$.
Otherwise, size coincides with density character.

The same is true for the category $\Met_{\cs}$ of $\cs$-sorted complete metric spaces (where the cardinality of an $\cs$-sorted space is the sum of the cardinalities of its sorts). The functor $V:\Met_\cs\to\Met$ sends an $\cs$-sorted metric space to the disjoint union
of its sorts. Clearly, $V$ preserves $\aleph_1$-directed colimits and sizes for any uncountable cardinal $\lambda$.

The forgetful functor $U_0:\Met\to\Set$ sends a complete metric space of density character $\lambda$ to a set
of cardinality $\leq\lambda^{\aleph_0}$. In particular, $U_0$ preserves all sizes $\lambda$ with $\lambda^{\aleph_0}=\lambda$.
}
\end{exam}

We now complete the category-theoretic description of mAECs, incorporating the underlying set functor $U:\ck\to\Set$.  Monomorphisms are clearly concrete, so in light of Remark~\ref{ldircolimsconcr} and Theorem~\ref{maecsacc}, we have:

\begin{theo}\label{maecsasconcr} For any mAEC $\ck$, $(\ck,U)$ is a coherent $\lsdk^+$-accessible category with directed colimits, concrete $\aleph_1$-directed colimits, and concrete monomorphisms. Moreover, it is iso-full in the sense of Remark~3.5 in \cite{LR}.

Finally, as in Example~\ref{metric}, $U$ preserves sizes $\lambda$ with $\lambda^{\aleph_0}=\lambda$.
\end{theo}

In fact, as we will see in Corollary~\ref{maecprespres}, an analysis of the functor $U$---particularly the fact that it preserves $\aleph_1$-directed colimits---ensures that it preserves a broader class of sizes, namely sufficiently large $\lambda$ with $\lambda^+\trianglerighteq\aleph_1$.

\section{$\kappa$-concrete AECs}\label{kaec}

We now introduce a category-theoretic framework, $\kappa$-concrete AECs, which generalize both AECs and mAECs, and, more broadly, any AEC-like classes where only sufficiently highly directed colimits are required to be concrete.  Moreover, we recall several notions from the broader theory of accessible categories that become indispensable in this context.  Chiefly, we recall the definition of the sharp inequality relation, $\trianglelefteq$.

\begin{defi}\label{klaecdef}{\em We say that a pair $(\ck,U)$ consisting of a category $\ck$ and faithful functor $U:\ck\to\Set$ is a \textit{$\kappa$-concrete AEC}, or \textit{$\kappa$-CAEC}, if
\begin{enumerate}\item $\ck$ is accessible with directed colimits, and all of its morphisms are monomorphisms.\\
\item $(\ck,U)$ is coherent, and has concrete monomorphims.\\
\item $(\ck,U)$ is replete and iso-full, in the sense of \cite{LR}~3.5.\\
\item $U$ preserves $\kappa$-directed colimits.\end{enumerate}}\end{defi}

Note that the only modification from the category-theoretic characterization of AECs in \cite{LR} comes in item (4), where we specify that only $\kappa$-directed colimits need be concrete.

\begin{rem}{\em While it will generally suffice to speak of $\kappa$-CAECs, we will occasionally need another parameter: when the underlying category $\ck$ of a $\kappa$-{\rm CAEC} is $\lambda$-accessible, we specify that $(\ck,U)$ is a {\it $(\kappa,\lambda)$-concrete AEC}, or \textit{$(\kappa,\lambda)$-{CAEC}.}}\end{rem}

In light of Theorem~\ref{maecsasconcr},

\begin{propo}Any mAEC $\ck$, equipped with its underlying set functor $U$, is an $\aleph_1$-{\rm CAEC}.  In particular, it is an $(\aleph_1,\lsdk^+)$-{\rm CAEC}.\end{propo}

\begin{rem}{\em We note that, for the purposes of this paper, we will have no need of repleteness or iso-fullness.  That is, we work in what one might call \textit{weak $\kappa$-CAECs}, which satisfy all of the conditions of Definition~\ref{klaecdef} except (3).  We note that the coherent accessible categories with concrete directed colimits considered in \cite{LR} are precisely the weak $\aleph_0$-CAECs.}\end{rem}

\begin{exams}\label{kaecexams}
{\em
(1) {\it Locally presentable categories with well-behaved regular monos:} Recall that a category is {\it locally presentable} just in case it is accessible and has all colimits.  Recall as well that we call a monomorphism {\it regular} if it is the equalizer of a pair of morphisms: note that in many cases, including $\Set$ (or any pretopos) and $\Ab$ (or any Abelian category), all monomorphisms are regular.  

Let $\ck$ be a locally presentable category in which regular monomorphisms are well-behaved, in the sense that
\begin{enumerate}
\item[(i)] If $gf$ and $g$ are regular monomorphisms then $f$ is a regular monomorphisms.
\item[(ii)] Regular monomorphisms commute with directed colimits. 
\end{enumerate}
The second condition means that given directed colimits $(a_i:A_i\to A)_{i\in I}$ and $(b_i:B_i\to B)_{i\in I}$ in $\ck$ and regular monomorphisms $f_i:A_i\to B_i$ for each $i\in I$, the map $\colim f_i:\colim A_i\to\colim B_i$ is itself a regular monomorphism. Under these assumptions, the subcategory $\ck_{reg}$ having the same objects as $\ck$ but regular monomorphisms as morphisms is closed under directed colimits in $\ck$. Following \cite{AR} Proposition 2.31 and Theorem 2.34, $\ck_{reg}$ is an accessible category. By \cite{AR} Theorem 5.30, moreover, there is a faithful functor $U:\ck\to\Set$ preserving regular monomorphisms and $\kappa$-directed colimits 
for some $\kappa$. This functor makes $\ck_{reg}$ a weak $\kappa$-CAEC.

(2) {\it Grothendieck topoi, Grothendieck categories:} Let $\ck$ be a locally presentable category with well-behaved regular monomorphisms and let $\cl$ be its full reflective subcategory closed under
$\kappa$-directed colimits for some $\kappa$ and such that the reflector $\cl\to\ck$ preserves regular monomorphisms. Then $\cl$ is itself a locally
presentable category with well-behaved regular monomorphisms. This covers two very broad classes of examples: recall that Grothendieck topoi and Grotendieck categories arise as such subcategories of, respectively, categories of presheaves $\Set^{\cc^{\op}}$ and categories of $R$-modules: the latter categories are locally presentable with well-behaved regular monomorphisms---as noted above, all of their monomorphisms are in fact regular.
 
(3) {\it Metric spaces with contractions:} We show that $\Met$, the category of complete metric spaces and contractions, is an $\aleph_1$-CAEC.  

Consider the single-sorted signature with binary relation symbols $R_r$ for each $0\leq r\in\reals$. Let $T$ consist of axioms
$$
(\forall x,y)(R_0(x,y)\leftrightarrow x=y)
$$
$$
(\forall x,y)(R_r(x,y)\rightarrow R_r(y,x))
$$
for all $r\leq s$
$$
(\forall x,y)(R_r(x,y)\rightarrow R_s(x,y))
$$
for all $r,s$
$$
(\forall x,y,z)(R_r(x,z)\wedge R_s(z,y)\rightarrow R_{r+s}(x,y))
$$
for $r_0\geq r_1\geq\dots r_n\geq\dots$ with $r=\lim r_n$
$$
(\forall x,y) (\bigwedge_n R_{r_n}(x,y)\rightarrow R_r(x,y))
$$
Since $T$ is a universal Horn theory in $L_{\omega_1,\omega}$, the category $\Mod(T)$ of $T$-models and homomorphisms is locally $\aleph_1$-presentable (see \cite{AR} 5.30). Regular monomorphisms are submodel embeddings and are well-behaved. If we interpret $R_r(a,b)$ as $d(a,b)\leq r$, $\Mod(T)$ is isomorphic to the category of generalized metric spaces
and contractions. Recall that, in a generalized metric space, the metric $d$ takes values in 
$\reals^{\ast}=\reals\cup\{\infty\}$. A contraction is a mapping $f$ such that $d(a,b)\geq d(fa,fb)$. Regular monomorphisms are isometries and are
well-behaved. The complete generalized metric spaces form a reflective full subcategory closed under $\aleph_1$-directed colimits and the reflector preserves isometries. Thus complete generalized metric spaces with isometries form a weak CAEC: in fact, a weak $\aleph_1$-CAEC. Since isomorphisms are isometries and isometries are preserved by the reflection, moreover, they form an $\aleph_1$-CAEC. By restriction, we obtain $\Met$ as an $\aleph_1$-CAEC. 

Notice that monomorphisms in $\Mod(T)$ are injective contractions, but they are not preserved by the reflection to complete generalized metric spaces.

(4) {\it Metric structures:} The procedure above works for any category of metric $L$-structures of the sort described in Definition~\ref{mlstruct}, assuming that the interpretations of the function and predicate symbols are not merely continuous, but contractions. Although this assumption removes the difficulty involving the nonexistence of certain directed colimits, highlighted in \ref{badstructs}, it is too restrictive: in Banach spaces, for example, the addition operation is continuous, but not a contraction.

(5) An ultrametric space is a special kind of a metric space in which the triangle inequality is replaced by
$$
d(x,y)\leq\max\{d(x,z),d(z,y)\}.
$$
In a generalized ultrametric space, the metric takes values in $\reals^\ast$. Morphisms of generalized ultrametric spaces are contractions.
We may axiomatize a generalized ultrametric space as in example (3) above, formalizing the ultrametric inequality as follows:
$$
(\forall x,y,z)(R_r(x,z)\wedge R_s(z,y)\rightarrow R_{\max\{r,s\}}(x,y))
$$
Thus the category of generalized ultrametric spaces is locally $\aleph_1$-presentable. The same is valid for its full subcategory $\ck$ consisting
of generalized ultrametric spaces where the metric takes values in $\{0,1,\dots,n,\dots,\infty\}$. This category $\ck$ coincides with the category
of $(\omega+1)^{\op}$-ultrametric spaces in the sense of \cite{A} and \cite{PR}, where $(\omega+1)^{\op}$ is the dual of the chain of ordinals $0,1,\dots,n,\dots,\omega$.
The category $\ck$ is a weak $\aleph_1$-CAEC. Analogously, the category $(\omega_1+1)^{\op}$-ultrametric spaces is a weak $\aleph_2$-CAEC and, moreover, it is not a weak $\aleph_1$-CAEC: colimits of $\omega_1$-chains fail to be concrete by an argument analogous to that for the failure of concreteness of colimits of $\omega$-chains of complete metric spaces. Similarly, for any regular cardinal $\aleph_\alpha$, the $(\omega_\alpha+1)^{\op}$-ultrametric spaces form a weak $\aleph_\alpha$-CAEC
which is not a weak $\aleph_\beta$-CAEC for any regular $\aleph_\beta<\aleph_\alpha$.

(6) {\it Grothendieck quasitopoi:} Similar to (1) above, we may consider locally presentable categories $\ck$ with well-behaved monomorphisms. In fact, the condition (i) is automatic, so we need only assume that monomorphisms commute with directed colimits.  Then $\ck_{mono}$, which has the same objects as $\ck$ and monomorphisms as morphisms, is a weak CAEC. Following \cite{GL}, any Grothendieck quasitopos
is a locally presentable category with well-behaved monomorphisms.
}
\end{exams}

In case $U$ preserves directed colimits, as in \cite{LR}, it also preserves $\lambda$-presentable objects for sufficiently large $\lambda$, i.e. $\lambda>\lambda_U$ for some $\lambda_U$ (see \cite{BR}~4.3).  Provided $(\ck,U)$ is coherent, this guarantees that $U$ in fact preserves sizes $\lambda>\lambda_U$, and not merely presentability.  Both statements fail if $U$ does not preserve directed colimits:

\begin{exam}{\em An object $M$ of $\Met$ is $\aleph_1$-presentable in $\ck$ if and only if it is separable, whereas $U(M)$ would be $\aleph_1$-presentable in $\Set$ if and only if it is countable.  Naturally, there are separable complete metric spaces that are not countable.}\end{exam}

We do, however, get a slightly weaker preservation result, which is a modification of \cite{BR}~4.3:

\begin{theo}\label{presentpreserv} Let $\ck$ be a $\lambda$-accessible category with $\kappa$-directed colimits, and let $F:\ck\to\cl$ be a faithful functor preserving $\kappa$-directed colimits and $\lambda$-presentable objects. Then $F$ preserves $\mu$-presentable objects for all $\mu$ with 
$\mu\triangleright\kappa$ and $\mu\geq\lambda$.\end{theo}

Before we proceed with the proof of Theorem~\ref{presentpreserv}, we recall the definition of the order relation $\trianglerighteq$, which first appeared in \cite{MP}.  It arises in response to a very natural question, namely: when is a $\lambda$-accessible category $\mu$-accessible, for $\mu>\lambda$?  The following result appears in \cite{MP} 2.3 (see also Theorem 2.11 in \cite{AR}):

\begin{theo}\label{slessequivthm} 
For regular cardinals $\lambda<\mu$, the following are equivalent:
\begin{enumerate}
\item Each $\lambda$-accessible category is $\mu$-accessible.
\item The category of $\lambda$-directed posets with order embeddings (which is $\lambda$-accessible) is $\mu$-accessible.
\item For each set $X$ of less than $\mu$ elements the poset of subsets of size less than $\lambda$, $P_{<\lambda}(X)$, has a cofinal set of cardinality less than $\mu$.
\item In each $\lambda$-directed poset, every subset of less than $\mu$ elements is contained in a $\lambda$-directed subset of less than $\mu$ elements.
\end{enumerate}\end{theo}

\begin{defi}\label{defsless} 
{\em For regular cardinals $\lambda$ and $\mu$, we say that \textit{$\lambda$ is sharply less than $\mu$}, denoted $\lambda\triangleleft\mu$, if they satisfy the equivalent conditions of the theorem above.}
\end{defi}

To make this a bit more concrete:

\begin{exams}\label{sless}{\em \begin{enumerate}
\item $\omega\triangleleft\mu$ for every uncountable regular cardinal $\mu$.
\item  For every regular $\lambda$, $\lambda\triangleleft\lambda^+$.  
\item For any regular cardinals $\lambda$ and $\mu$ with $\lambda\leq\mu$, $\lambda\triangleleft(2^\mu)^+$.
\item Whenever $\mu$ and $\lambda$ are regular cardinals with $\beta^\alpha<\mu$ for all $\beta<\mu$ and $\alpha<\lambda$, then $\lambda\triangleleft\mu$.  
\end{enumerate}}\end{exams}

See 2.3 in \cite{MP} or 2.13 in \cite{AR} for more examples.  We note that for $\mu$ sufficiently large relative to $\lambda$, the relation $\lambda^+\triangleleft\mu^+$ is equivalent to the more fundamental equality $\mu^{\lambda}=\mu$ (we thank the anonymous referee for providing a proof of this fact in the case $\lambda=\aleph_0$).

\begin{propo}\label{sharpcard} Let $\mu^+>2^\lambda$.  Then $\mu^+\triangleright\lambda^+$ if and only if $\mu^\lambda=\mu$.\end{propo}
\begin{proof}
The ``if'' direction follows immediately from \ref{sless}(4), and holds even in case $\mu^+\leq 2^\lambda$.  Going the other way, let $\mu$ be such that $\mu^\lambda>\mu$, and suppose that $\mu^+\triangleright\lambda^+$.  Then $P_{\leq\lambda}(\mu)$ contains a cofinal subset of size $\mu$, say $\{X_\alpha\,|\,\alpha<\mu\}$.  As $\mu^\lambda>\mu$, by assumption, we may choose a family of $\mu^+$ distinct elements of $P_{\leq\lambda}(\mu)$, $\{Y_\beta\,|\,\beta<\mu^+\}$.  Let $f:\mu^+\to\mu$ satisfy $Y_\beta\subseteq X_{f(\beta)}$.  By the pigeonhole principle, $\mu^+$ of the $Y_\beta$ must belong to the same $X_\alpha$; that is, some $X_\alpha$ must contain $\mu^+$ distinct subsets of size at most $\lambda$.  But $X_\alpha$, being of size at most $\lambda$, can contain at most $\lambda^\lambda=2^\lambda$ such subsets: by assumption, $\mu^+>2^\lambda$.
\end{proof}

\begin{rem}\label{GCHsharp}
{
\em
Under GCH, these relations correspond exactly: $\mu^+\triangleright\lambda^+$ if and only if $\mu^\lambda=\mu$.  Otherwise, $\mu^+\triangleright\lambda^+$ will hold on an initial segment of cardinals $\mu^+\leq 2^\lambda$, namely $\mu=\lambda^+$ and all of its successors below $2^\lambda$.
}
\end{rem}

\begin{rem}\label{stronger} {\em Let $\ck$ be a $\lambda$-accessible category with $\kappa$-directed colimits, $\mu\triangleright\kappa$ and $\mu\geq\lambda$. Analogously
to \cite{BR}~4.1, we show that $\ck$ is $\mu$-accessible.

Given an object $K$ of $\ck$, there is a $\lambda$-directed colimit $(a_i:A_i\to K)_{i\in I}$ of $\lambda$-pre\-sen\-table objects $A_i$. 
Let $\hat{I}$ be the poset of all $\kappa$-directed subsets of $I$ of cardinalities less than $\mu$ (ordered by inclusion). Since every subset of $I$ 
having less than $\mu$ elements is contained in a $\kappa$-directed subset of $I$ having less than $\mu$ elements (cf. \ref{slessequivthm}), clearly, 
$\hat{I}$ is $\mu$-directed.  For each $M\in\hat{I}$, let $B_M$ be a colimit of the subdiagram indexed by $M$. Then $B_M$ is $\mu$-presentable.  $K$ is 
a $\mu$-directed colimit of the $B_M$, $M\in\hat{I}$.  Thus $\ck$ is $\mu$-accessible.}\end{rem}

\begin{proof}(Theorem~\ref{presentpreserv}) Let $K$ be $\mu$-presentable object of $\ck$. 
Following \ref{stronger}, $K$ is a $\mu$-directed colimit of objects $B_M$ where each $B_M$ is a $\mu$-small $\kappa$-directed colimit of $\lambda$-presentable objects. Since $K$ is $\mu$-presentable, it is a retract of some $B_M$.  Since $F(B_M)$ is a $\mu$-small $\kappa$-directed colimit of $\lambda$-presentable objects, $F(B_M)$ is $\mu$-presentable in $\cl$. Thus $F(K)$ is $\mu$-presentable in $\cl$ as a retract of $F(B_M)$. We have proved that $F$ preserves 
$\mu$-presentable objects.
\end{proof}

In light of the theorem, it is important to establish that there are, in fact, cardinals $\lambda$ such that $U$ preserves $\lambda$-presentable objects.

\begin{rem} 
{
\em
If $\ck$ is a $(\kappa,\lambda)$-CAEC, then in particular $U:\ck\to\Set$ is an accessible functor and, by Theorem 2.19 in \cite{AR}, is 
$\theta$-accessible (that is, $\ck$ is $\theta$-accessible and $U$ preserves $\theta$-directed colimits) and preserves $\theta$-presentable objects for some cardinal $\theta$.
}\end{rem}

\begin{defi}Let $\ck$ be a $(\kappa,\lambda)$-CAEC.  We define $\lambda_U$ to be the least cardinal such that $U$ is $\lambda_U$-accessible and preserves $\lambda_U$-presentable objects.\end{defi}

\begin{rem}\label{pressize} 
{
\em 
This also applies in case $\ck$ is an mAEC, of course, and we get a very straightforward upper bound on $\lambda_U$, namely $(\lsdk^{\aleph_0})^+$. 
}
\end{rem}

As a matter of convention, we insist that $\lambda_U\geq\lambda$ in a general $(\kappa,\lambda)$-CAEC, hence we require $\lambda_U\geq\lsdk^+$ in an mAEC $\ck$.

\begin{coro}\label{klaecprespres}Given any $(\kappa,\lambda)$-{\rm CAEC} $(\ck,U)$, $U$ preserves $\mu$-presentable objects for all $\mu\triangleright\kappa$ 
with $\mu\geq\lambda_U$.\end{coro}

We can easily rewrite this statement in terms of size, rather than presentability: Corollary~\ref{klaecprespres} asserts that the forgetful functor $U:\ck\to\Set$ preserves objects of size $\mu$ with $\mu^+\triangleright\kappa$ and $\mu^+\geq\lambda_U$.  In the special case of mAECs, this becomes:

\begin{coro}\label{maecprespres} If $\ck$ is an mAEC, its underlying set functor $U$ preserves sizes $\mu$ with $\mu^+\triangleright\aleph_1$ and $\lambda_U^+$.\end{coro}

\begin{rem}\label{aleph1spec} 
{
\em
(1) Following Proposition~\ref{sharpcard}, the condition that $\mu^+\triangleright\aleph_1$ is equivalent to the more familiar condition $\mu^{\aleph_0}=\mu$ for $\mu^+>2^{\aleph_0}$.  Hence in the worst case scenario, namely $\lambda_U=(\lsdk^{\aleph_0})^+$, Corollary~\ref{maecprespres} guarantees preservation only in $\mu$ with $\mu^{\aleph_0}=\mu$.  Given a smaller $\lambda_U$, though, and barring the assumption of GCH, $\mu^+\triangleright\lambda^+$ may hold of certain small cardinals for which $\mu^{\aleph_0}>\mu$---see Remark~\ref{GCHsharp}.

(2) The spectrum of $\mu$ with $\mu^+\triangleright\aleph_1$ contains gaps, as one would expect.  In particular, $\aleph_\omega$ does not belong to this class.  Although a detailed argument is given \cite{AR} 2.18(8), the basic details are instructive: consider a set $X=\bigcup_{n<\omega} X_n$ with each $X_n$ of cardinality $\aleph_n$.  One can check that no subset of X of $P_{\aleph_1}(X)$ of cardinality less than $\aleph_{\omega+1}$ is cofinal in $P_{\aleph_1}(X)$. Hence $\aleph_{\omega+1}$ and $\aleph_1$ do not satisfy the third of the equivalent conditions in Theorem~\ref{slessequivthm}, meaning that $\aleph_{\omega+1}\not\triangleright\aleph_1$.
}\end{rem}

\section{Presentation Theorem, EM-Models}

An essential result from the theory of AECs is Shelah's Presentation Theorem, which shows that any AEC can be represented as a PC-class, i.e. given any AEC $\ck$ in signature $L$, there is an extension $L'$ of $L$, a first-order $L'$-theory $T'$ and a set of $T'$-types $\Gamma$ so that
$$\ck=\{ M'\!\upharpoonright\!L \,\,|\,\, M'\models T', M\mbox{ omits }\Gamma\}$$
Indeed, the theorem asserts more, namely that the reduct $\upharpoonright\!\!L$ is functorial from $\ck'=\{ M' \,\,|\,\, M'\models T', M'\hbox{ omits }\Gamma\}$ to $\ck$, and one can easily see that it is faithful, surjective on objects, and preserves directed colimits.  Remark~2.6 in \cite{LR} yields a substantial generalization of this fact, namely the following categorical presentation theorem:

\begin{theo}\label{cover}Let $\ck$ be an accessible category with directed colimits and whose morphisms are monomorphisms.  Then there is a finitely accessible category $\ck'$ and a functor $F:\ck'\to\ck$
that is faithful, surjective on objects, and preserves directed colimits.\end{theo}

We note that this Presentation Theorem is loosely analogous to Shelah's: the category $\ck'$ is the closure under directed colimits of the subcategory of $\lambda$-presentable objects that generate $\ck$, in which those $\lambda$-presentable objects become finitely presentable.  Shelah's Presentation Theorem, likewise, hinges on a coding of the models of cardinality $\lsk$ by finite tuples.  There are significant differences, however.  Theorem~\ref{cover} holds in significantly greater generality (among other things, it does not require coherence of $(\ck,U)$; more importantly, it also covers mAECs and, indeed, $(\kappa,\lambda)$-CAECs).  Moreover, the closure under directed colimits involved here is a universal construction, by contrast with the somewhat \emph{ad hoc} flavor of the construction of $\ck'$ in Shelah's result.

In any case, this applies to $\kappa$-CAECs, by definition:

\begin{coro}\label{kaecpresthm}For any $\kappa$-CAEC $(\ck,U)$, there is a finitely accessible category $\ck'$ and a functor $\ck'\stackrel{F}{\to}\ck$
that is faithful, surjective on objects, and preserves directed colimits.\end{coro}

As a special case, we get the promised analogue of Shelah's Presentation Theorem for mAECs:

\begin{coro}\label{presthm}For any mAEC $\ck$, there is a finitely accessible category $\ck'$ and a functor
$$\ck'\stackrel{F}{\to}\ck$$
that is faithful, surjective on objects, and preserves directed colimits.\end{coro}

The latter enters a crowded field of presentation theorems, including Fact~5.1 in \cite{HH}, which supports the explicit description of an EM-functor but holds only in the case of homogeneous mAECs in a countable signature.  This has been extended to general mAECs in both a continuous and a discrete version.  


The discrete version, Corollary~6.3 in \cite{Bpres}, is obtained by an ingenious process---passing from an mAEC $\ck$ to an auxiliary (discrete) AEC $\ck_{\rm dense}$ consisting of dense substructures of its models.

\begin{theo}[Boney] Let $\ck$ be an mAEC in signature $L$. Then there is a (discrete) language $L_1$ of size $\lsdk$, an $L_1$-theory $T_1$, and a set of $T_1$-types $\Gamma$ such that 
$$\ck=\{\overline{M_1\!\upharpoonright_{L}}\,|\, M_1\models T_1, M\mbox{ omits }\Gamma\}$$
where the completion is taken with respect to a canonically definable metric.\end{theo}

In each case, the result can be rewritten to resemble our Corollary~\ref{presthm}, the crucial difference being the nature of the category $\ck'$ used to cover the mAEC $\ck$.  By allowing ourselves a certain flexibility in our choice of presentation functor $F$---not necessarily a reduct but, by virtue its accessibility, nonetheless susceptible to a logical characterization---we are able to find $\ck'$ finitely accessible.  As for AECs in \cite{LR}, this is enough to guarantee the existence of an EM-functor for any large $\kappa$-CAEC.

\begin{theo}\label{kaecemfunct}Let $(\ck,U)$ be a large $\kappa$-CAEC.  Then there is a faithful functor $E:\Lin\to\ck$ that preserves directed colimits and, moreover, there is a cardinal $\lambda_E$ such that $E$ preserves all sizes $\lambda$ with $\lambda^+\geq\lambda_E$.\end{theo}

\begin{proof}See Corollary~2.7 in \cite{LR}.\end{proof}

In fact, we can give a clear bound on $\lambda_E$:

\begin{propo}\label{bound} Let $(\ck,U)$ be a $(\kappa,\lambda)$-CAEC, and let $\mu$ be the number of morphisms among objects of $\ck$ of size $\lambda$. Then $\lambda_E$ can be taken to be $(2^\mu)^+$.\end{propo}

\begin{proof} In fact, the finitely accessible category $\ck'$ from \ref{cover} is $\Ind(\cc)$ where $\cc$ is the full subcategory of $\ck$ consisting of objects
of size $LS^d(\ck)$. Thus it can be axiomatized by a basic theory $T$ of $L_{\mu^+,\omega}(\Sigma)$ (see \cite{AR} 5.35). Let $T^\ast$ be the skolemization 
of $T$ given by adding operation symbols $f_\varphi$ for each formula $\varphi=(\exists x)\psi$, where $\psi$ has free variables $x_1,\dots,x_n$, and formulas
$$
(\forall x_1,\dots,x_n)(\varphi(x_1,\dots,x_n)\to\psi(f_\varphi(x_1,\dots,x_n),x_1,\dots,x_n).
$$
Let $\Sigma^\ast$ be the resulting signature.
This skolemization has the property that each subset $X$ of a $T^\ast$-model $M$ generates the smallest $T^\ast$-model containing $X$ and being included
in $M$. Then the EM-functor $\Lin\to\Mod(T^\ast)$ sends finite chains to finitely generated submodels of a suitable $T^\ast$-model $M$. Since 
$L_{\mu^+,\omega}(\Sigma^\ast)$ has $2^\mu$ formulas, these finitely generated models have size $2^\mu$.\end{proof}

One might think of this $\lambda_E$ as an analogue of the kind of minimal bloating one expects when passing from a set of indiscernibles to its Skolem hull in more traditional accounts of EM-models.  Noting that any mAEC with arbitrarily large models is a large $\aleph_1$-CAEC, we have:

\begin{coro}\label{emfunct}Let $\ck$ be an mAEC with arbitrarily large models.  Then there is a faithful functor $E:\Lin\to\ck$ that preserves directed colimits and, moreover, there is a cardinal $\lambda_E$ such that $E$ preserves all sizes $\lambda$ with $\lambda^+\geq\lambda_E$.\end{coro}

Again, the bound of Proposition~\ref{bound} applies to mAECs as well, meaning that the size of $\lambda_E$ is controlled by the number of morphisms between models of density character $\lsdk$.  This generalizes the existence result \cite{HH} 5.7, which holds only in the countable, homogeneous case. The syntactic presentation in \cite{Bpres} should yield a comparable result for general mAECs, with a potentially smaller $\lambda_E$: in particular, it should suffice to take $\lambda_E=\lsdk$.  

\section{Stability}\label{stabsection}

We now turn our attention to the relationship between categoricity and Galois-stability in mAECs.  Recall that the definition of Galois types in this context exactly matches the definition in AECs, although for our purposes it is more convenient to replace incidences of $\ksst$ with $\ck$-embeddings, and to make the superficial generalization to $(\kappa,\lambda)$-CAECs:

\begin{defi} {\em Let $\ck$ be a $(\kappa,\lambda)$-CAEC.  For any $M\in\ck$, we define a relation on pairs $(f,a)$, where $f:M\to N$ is a $\ck$-embedding and $a\in U(N)$, as follows: given $(f_1,a_1)$ and $(f_2,a_2)$ with $f_i:M\to N_i$, $(f_1,a_1)\sim (f_2,a_2)$ if and only if there is $N\in\ck$ and embeddings $g_i:N_i\to N$ such that $g_1f_1=g_2f_2$ and $U(h_1)(a_1)=U(h_2)(a_2)$.}\end{defi}

Assuming the amalgamation property, this is an equivalence relation.  By a {\em Galois type over $M\in\ck$}, we mean an equivalence class of such pairs.  If, in addition, we assume joint amalgamation and the existence of arbitrarily large models, we have recourse to a {\em monster model} $\monst$ in $\ck$.  Per \cite{LR}~4.3, we may simply identify Galois types over an object $M$ with orbits in $\monst$ fixing $M$, as in AECs (see, e.g., \cite{Ba}~8.9) and mAECs (see, e.g., \cite{Z}~1.3.7).

While the standard treatment in AECs considers the set of types over a model $M$ as a discrete set---or possibly a topological space, as in \cite{Ltop}---this is not in the spirit of mAECs.  Associating types with orbits of elements in $\monst$, which is itself a metric structure, we obtain a pseudometric on the set of types over $M$: let $\metr$ be the infimum of the distances between elements of respective orbits.  In fact, this pseudometric becomes a metric on the set of types, assuming what is variously known as the {\em perturbation property} (see \cite{HH}~2.12) or {\em continuity of types property} (see \cite{VZ}~2.9): for any $M\in\ck$ and convergent sequence $(a_n)\to a$ in $\monst$, if $\gatp(a_n/M)=\gatp(a_0/M)$ for all $n$, then $\gatp(a/M)=\gatp(a_0/M)$.  With respect to the metric $\metr$, one can give metric refinements of Galois-type-theoretic notions familiar from AECs, as we will see momentarily.  

In the results that follow, we make the following blanket assumptions:

\begin{assume}\label{assume} {\em Henceforth all $(\kappa,\lambda)$-CAECs are assumed to be large and to satisfy the joint embedding and amalgamation properties.  All mAECs are assumed to contain arbitrarily large models, and to satisfy the joint embedding, amalgamation, and perturbation properties.}\end{assume}

We focus first on stability, which we present here in two forms: a discrete version attuned to the general case of $(\kappa,\lambda)$-CAEC, and the fully metric version best suited to mAECs.  Following \cite{Z}, we restrict ourselves to types over models, rather than sets.

\begin{defi} {\em \begin{enumerate}
\item Let $(\ck,U)$ be a $(\kappa,\lambda)$-CAEC.  We say that it is \textit{$\mu$-stable} if for all $M\in\ck$ of size $\mu$ (in the sense of $\ck$), $\gat(M))\leq\mu$.
\item Let $\ck$ be an mAEC.  We say that $\ck$ is \textit{$\mu$-$\metr$-stable} if for any $M$ in $\ck$ with $\dc(M)=\mu$, then $\dc(\gat(M))\leq\mu$.\end{enumerate}}\end{defi}

That is, $\ck$ is $\mu$-$\metr$-stable if for any $M$ of size $\mu$ in the sense of $\ck$, the set of types over $M$ is of size at most $\mu$ in the sense of $\Met$, the category of complete metric spaces and contractions.  It may be significant that this exactly matches the discrete version of $\mu$-stability, but with $\Met$ in place of $\Set$.  Indeed, this points the way to a broader project: to properly engage with $\metr$-stability, $\metr$-tameness, and other related notions, it would be beneficial to forget less structure than we do with a discretizing functor $U:\ck\to\Set$---the authors have already made progress along these lines.  For the present, though, we restrict ourselves to the tools afforded by $(\kappa,\lambda)$-CAEC theory.  As $\mu$-stability in the discrete sense certainly implies $\mu$-$\metr$-stability, they are more than sufficient.

We now consider the question of stability of a $(\kappa,\lambda)$-CAEC below a categoricity cardinal.  Note that we define categoricity in the sense of the underlying category $\ck$: we say that it is $\nu$-categorical if it contains (up to isomorphism) exactly one object of size $\nu$, i.e. of presentability rank $\nu^+$.  

\begin{rem}{\em In case $(\ck,U)$ arises from an mAEC, $\nu$-categoricity asserts the existence of a unique model of density character (rather than cardinality) $\nu$.  This corresponds to the definition of $\nu$-$\metr$-categoricity in \cite{VZ}}.\end{rem}

The central theorem is the following, which is adapted from Theorem~7.4 in \cite{LR}.  Recalling Assumption~\ref{assume}, we have:

\begin{theo}\label{stabbelowcat}Let $(\ck,U)$ be a $(\kappa,\lambda)$-CAEC. If $\ck$ is $\nu$-categorical, then $\ck$ is $\mu$-Galois stable for all $\lambda_U+\lambda_E\leq\mu^+\leq\nu$ with $\mu^+\triangleright\kappa$.\end{theo}

Notice the appearance of the sharp inequality relation: unlike in \cite{LR}, $U$ does not preserve all sufficiently large sizes, but rather those sharply larger than $\kappa$.  This extra piece of bookkeeping is the only essential change from Theorem~7.4 in \cite{LR}, and we give an argument for Theorem~\ref{stabbelowcat} above that parallels the proof of that theorem (and, in turn, the AEC-centric argument of \cite{Ba1}), indicating only the areas where modifications are required. 

\begin{lemma}Let $(\ck,U)$ be a $(\kappa,\lambda)$-CAEC and $E:\Lin\to\ck$ an EM-functor as in Corollary~\ref{emfunct}.  If $I$ brimful as a linear order, then $E(I)$ is brimful in $\ck$.\end{lemma}
\begin{proof} The argument for the analogous result in \cite{LR}, Lemma~7.7, makes no use of the functor $U$, hence goes through without change.\end{proof}

\begin{lemma}Let $(\ck,U)$ be a $(\kappa,\lambda)$-CAEC and $E:\Lin\to\ck$ an EM-functor.  If $\ck$ is $\nu$-categorical, then for all $\mu$ with $\lambda_U+\lambda_E\leq\mu^+\leq\nu$ and $\mu^+\triangleright\kappa$, the unique object $M$ of size $\nu$ is $\mu$-stable: that is, there are at most $\mu$ Galois types over any $M_0$ of size $\mu$ admitting an embedding $M_0\to M$.\end{lemma}
\begin{proof} The proof of the analogue in \cite{LR}, Lemma~7.8, involves showing that any type over an object $M_0$ of size $\mu$ embeddable in the categorical object $M$ is in fact realized in an intermediate $\mu$-universal structure $\bar{M}$ that is also of size $\mu$.  The conclusion that there are at most $\mu$ types over $M_0$ follows from the fact that $|U(\bar{M})|\leq\mu$.  In the present context, this follows from Corollary~\ref{klaecprespres} and the assumption that $\mu^+\triangleright\kappa$.\end{proof}

We complete the proof of Theorem~\ref{stabbelowcat} precisely as in \cite{LR}---the argument there requires only that $U$ preserve $\mu^+$-directed colimits, and we here assume that $\mu^+\triangleright\kappa$, hence also $\mu^+>\kappa$.

As a special case, we have:

\begin{theo}\label{mstabbelowcat} Let $\ck$ be an mAEC.  If $\ck$ is $\nu$-categorical (i.e. $\nu$-$\metr$-categorical), then it is $\mu$-stable (hence $\mu$-$\metr$-stable) for all $\lambda_U+\lambda_E\leq\mu^+\leq\nu$ with $\mu^+\triangleright\aleph_1$.\end{theo}
\begin{proof} Since $\ck$, equipped with its underlying set functor $U$, forms an $(\aleph_1,\lsdk^+)$-CAEC, Theorem~\ref{stabbelowcat} implies $\mu$-stability for all $\lambda_U+\lambda_E\leq\mu^+\leq\nu$ with $\mu^+\triangleright\aleph_1$.  Naturally, $\mu$-stability implies $\mu$-$\metr$-stability.\end{proof}

We contrast this with Corollary~5.8 in \cite{HH} and the remark that immediately follows it: there EM-models are used to show that for any mAEC $\ck$ of the form considered here---with amalgamation, joint embedding, perturbation---and, in addition, with $\lsdk=\aleph_0$, if $\ck$ is $\nu$-$\metr$-categorical for some uncountable $\nu$, then $\ck$ is $\mu$-$\metr$-stable in all $\mu$ such that $\mu^{\aleph_0}=\mu$.  Theorem~\ref{mstabbelowcat} applies to $\ck$ with arbitrary $\lsdk$ and, in any case, yields stability in a potentially larger assortment of cardinals $\mu$.  

Recall that $\lambda_U\leq(\lsdk^{\aleph_0})^+$, and thus $\lambda_U\leq (2^{\aleph_0})^+$ in case $\lsdk=\aleph_0$.  We note that if this bound is tight, if $\lambda_E$ is at least $(2^{\aleph_0})^+$, or if GCH is assumed, Remark~\ref{aleph1spec}(1) implies that the cardinals $\mu$ for which $\ck$ is $\mu$-$\bf{d}$-stable, i.e. those $\mu\geq\lambda_U+\lambda_E$ with $\mu^+\triangleright\aleph_1$, are precisely those with $\mu^{\aleph_0}=\mu$, just as in \cite{HH}.  Otherwise, $\mu^+\triangleright\aleph_1$ is a strictly weaker condition.



\section{Saturated Models}\label{satmods}

Recall that an object $K$ in a given category is said to be $\lambda$-saturated if for any $\lambda$-presentable objects $M$ and $N$ and morphisms $g:M\to K$ and $f:M\to N$ there is $h:N\to M$ so that $hf=g$.  We now relate this notion to Galois-saturation, in the sense of \cite{LR} 6.1:

\begin{defi}{\em Let $(\ck,U)$ be an accessible category with directed colimits. We say that a type $(f,a)$ where $f:M\to N$ is \textit{realized in $\ck$} if there is a morphism $g : M \to \ck$ and $b\in U(\ck)$ such that $(f, a)$ and $(g, b)$ are equivalent.\\
Let $\lambda$ be a regular cardinal. We say that K is \textit{$\lambda$-Galois saturated} if for any $g : M \to K$ where $M$ is $\lambda$-presentable and any type $(f,a)$ where $f :M \to N$ there is $b\in U(K)$ such that $(f,a)$ and $(g,b)$ are equivalent.}\end{defi}

\begin{rem}{\em If $(\ck,U)$ arises from an mAEC, this definition corresponds to the $\lambda$-$\metr$-saturation mentioned in \cite{VZ}.}\end{rem}

Recalling Assumption~\ref{assume},

\begin{theo}Let $(\ck,U)$ be a $(\kappa,\lambda)$-CAEC.  For regular cardinals $\nu$ with $\nu^+\triangleright\kappa$ and $\nu^+\geq\lambda_U+\lambda_E$, an object $K\in\ck$ is $\nu$-Galois saturated if and only if it is $\nu$-saturated.\end{theo}
\begin{proof} The proof proceeds exactly as in the proof of Theorem 6.2 in \cite{LR}, except in two details: first, in the limit stages of the inductive construction of the map $h:N\to K$ witnessing the $\nu$-saturation of $K$.  Because our category is closed under directed colimits we may simply take colimits at limit stage $j$, just as in that proof, but potential nonconcreteness of the directed colimits forces us to be slightly more careful with the increasing chain of partial set-embeddings $t_i$ of $U(N)$ into $U(M_i)$---here we cannot simply take $t_j=\cup_{i<j} t_i$ given that the codomain, $\colim U(M_i)$, need not correspond to the desired $U(\colim M_i)$.  Fortunately, we may obtain $t_j$ by composition of $\cup_{i<j}t_i$ with the canonical map $\colim U(M_i)\to U(\colim M_i)$ which, by concreteness of monomorphisms in $\ck$ and a short diagram chase, is an injection.  With this modification, the rest of the construction can be carried out as before.

The only other change is that the fact that $|U(N)|\leq\nu$, which is needed for the enumeration at the heart of the inductive construction, now follows from the preservation by $U$ of $\nu$-presentable objects with $\nu$ satisfying the inequalities in the statement of the theorem, rather than the stronger eventual preservation result used in \cite{LR}.\end{proof}

We note that certain results on the existence of saturated models in \cite{LR} also hold, albeit weakened through the introduction of the sharp inequality condition.

\begin{propo}Let $\ck$ be a $(\kappa,\lambda)$-CAEC.  If $\ck$ is $\nu^+$-categorical for $\nu^{+}\geq\lambda_U+\lambda_E$ and $\nu^{+}\triangleright\kappa$, then the unique object of size $\nu^+$ is saturated.\end{propo}
\begin{proof}As in \cite{LR}, we choose a model $M_0$ of size $\nu$ (notice that such an object exists, by Theorem~\ref{kaecemfunct}) and build a continuous chain $\langle M_i\,|\,i<\nu^+\rangle$ where each $M_i$ is of size $\nu$ and $M_{i+1}$ realizes all types over $M_i$.  By Theorem~\ref{stabbelowcat}, $\ck$ is $\nu$-stable, so the successor step from $M_i$ to $M_{i+1}$ is easily accomplished.  At limit stages, we take colimits.  Potential nonconcreteness is not an issue, although we might worry about the final stage,
$$M_{\nu^+}=\colim_{i<\nu^+} M_i$$
Fortunately, $\nu^+$ is regular and strictly larger than $\kappa$, so this colimit is in fact concrete: $U(M_{\nu^+})=\bigcup_{i<\nu^+}U(M_i)$.\end{proof}

In fact, this proof establishes more:

\begin{propo}Let $\ck$ be a $(\kappa,\lambda)$-CAEC.  If $\ck$ is $\nu$-categorical, then for any $\lambda_U+\lambda_E\leq\mu^+<\nu$ with $\mu^+\triangleright\kappa$, $\ck$ contains a saturated object of size $\mu^+$.\end{propo}

For mAECs, in particular, this gives many saturated models.  Recalling, again, that $\lambda_U\leq(\lsdk^{\aleph_0})^+$ for any mAEC $\ck$, we have:

\begin{coro}Let $\ck$ be an mAEC.  If $\ck$ is $\nu$-categorical, then for any $\lambda_U+\lambda_E\leq\mu^+<\nu$ with $\mu^+\triangleright\aleph_1$, $\ck$ contains a saturated object of size $\mu^+$.\end{coro}

While this result misses limit cardinals sharply larger than $\aleph_1$, Theorem~\ref{stabbelowcat} also allows us to infer the existence of limit models in such cases, which can, in certain circumstances, stand in for saturated models.  In particular:

\begin{propo}Let $\ck$ be an mAEC.  If $\ck$ is $\nu$-categorical, then for any $\mu$ with $\lambda_U+\lambda_E\leq\mu^+<\nu$ and $\mu^+\triangleright\aleph_1$, and any $M\in\ck$ of size $\mu$, there is a limit model $M'$ over $M$ which is also of size $\mu$.\end{propo}
\begin{proof} Theorem~\ref{stabbelowcat} guarantees $\mu$-$\metr$-stability in all such $\mu$.  The result then follows from Corollary 3.7 in \cite{VZ}.\end{proof}

\section{Acknowledgements}

The first author wishes to acknowledge John Baldwin and Will Boney for highly useful discussions on the sidelines of the Beyond First Order Model Theory special session at the 2015 Joint Math Meetings.  The authors would also like to acknowledge the comments of the anonymous referee, which have led to substantial improvements in this text.

\end{document}